\documentclass[12pt]{article}

\usepackage{geometry}
\usepackage{pifont}
\usepackage{soul}

\usepackage{enumitem}

\setenumerate{label=(\alph*)}
\setitemize{label=\scriptsize{\ding{118}}}

\usepackage[cmex10]{amsmath}
\usepackage{amsthm, amssymb}
\usepackage{booktabs,cite}

\usepackage{psfrag}
\usepackage{mathtools}

\newcommand{\imi}{\mathsf{i}}

\usepackage{hyperref}
\usepackage{breakurl}
\allowdisplaybreaks

\theoremstyle{theorem}
\newtheorem{theorem}{Theorem}[section]
\newtheorem{lemma}[theorem]{Lemma}

\newtheorem{corollary}[theorem]{Corollary}
\newtheorem{alg}[theorem]{Algorithm}

\theoremstyle{definition}
\newtheorem{definition}[theorem]{Definition}

\usepackage{framed}

\numberwithin{equation}{section}
\numberwithin{figure}{section}
\newcommand{\coloneqq}{:=}
\newcommand{\spar}{\boldsymbol{\Phi}}

\newcommand{\R}{\mathbb{R}}

\newcommand{\Z}{\mathbb{Z}}

\newcommand{\sph}{S}
\newcommand{\hatK}{\hat{K}}
\newcommand{\K}{K}
\newcommand{\F}{f}
\newcommand{\G}{g}
\newcommand{\domm}{\Delta}
\newcommand{\To}{\mathcal T}
\newcommand{\Vo}{\mathcal{V}}

\newcommand{\dr}{\mathrm{d}r}
\newcommand{\dom}{\mathrm{d}\omega}
\newcommand{\dt}{\mathrm{d}t}
\newcommand{\dx}{\mathrm{d}x}
\newcommand{\suppp}{\mathrm{supp}}

\newcommand{\fn}{\mathbf{f}}
\newcommand{\gn}{\mathbf{g}}
\DeclareMathOperator{\Kn}{\mathbf{K}}
\DeclareMathOperator{\In}{\mathbf{I}}

\makeatletter
\newcommand*\bigcdot{\mathpalette\bigcdot@{.6}}
\newcommand*\bigcdot@[2]{\mathbin{\vcenter{\hbox{\scalebox{#2}{$\m@th#1\bullet$}}}}}
\makeatother
\newcommand\inner[2]{{#1}\bigcdot {#2}}

\newcommand{\D}{D_\RR}

\newcommand{\trans}{\mathsf{T}}
\newcommand{\rmd}{\mathrm d}
\newcommand*\dd{\mathop{}\!\mathrm{d}}
\newcommand{\eps}{\epsilon}
\newcommand{\ph}{\varphi}

\newcommand\abs[1]{\left\vert#1\right\vert}
\newcommand\sabs[1]{\lvert#1\rvert}
\newcommand\norm[1]{\left\lVert#1\right\rVert}
\newcommand\snorm[1]{\lVert#1\rVert}

\newcommand\set[1]{\left\{#1\right\}}

\newcommand\sset[1]{\{#1\}}
\newcommand{\om}{\omega}
\newcommand{\al}{\alpha}
\newcommand{\la}{\lambda}
\newcommand{\RR}{R}
\newcommand{\pp}{p}
\newcommand{\qq}{q}
\newcommand{\PP}{P}
\newcommand{\QQ}{Q}
\newcommand{\ii}{i}
\newcommand{\MM}{M}
\newcommand{\NN}{N}

\newcommand{\kl}[1]{\left(#1\right)}
\newcommand{\bkl}[1]{\left(#1\right)}

\numberwithin{equation}{section}
\usepackage{units}

\begin{document}

\title{Inversion of the attenuated V-line transform for SPECT with Compton cameras}

\renewcommand{\thefootnote}{\fnsymbol{footnote}}
\author{Markus Haltmeier\footnotemark[2] \and Sunghwan Moon\footnotemark[3] \and Daniela Schiefeneder\footnotemark[2]}

\date{\small \footnotemark[2] Department of Mathematics, University of Innsbruck\\
Technikerstrasse 13, A-6020 Innsbruck, Austria\\
  { \tt \{Daniela.Schiefeneder,Markus.Haltmeier\}@uibk.ac.at}\\[1em]
\footnotemark[3] Department of Mathematical Sciences, Ulsan National Institute of Science and Technology\\
Ulsan 44919, Republic of Korea.\\
{\tt shmoon@unist.ac.kr}}

\markboth{The attenuated V-line transform for SPECT with Compton cameras}%
{The attenuated V-line transform for SPECT with Compton cameras}

\maketitle

\begin{abstract}
The Compton camera is  a promising alternative to the Anger camera for imaging gamma radiation,  with the potential to significantly increase the sensitivity of SPECT. Two-dimensional Compton camera image reconstruction can be implemented by inversion of the V-line transform, which integrates the emission distribution over V-lines (unions of two half-lines), that have vertices on a surrounding detector array.
Inversion of the V-line transform  without attenuation has recently been addressed by several authors. However, it is well known from standard SPECT that ignoring attenuation can significantly degrade the quality of the reconstructed image. In this paper we address this issue and study  the attenuated V-line transform accounting for attenuation of photons in SPECT with Compton cameras. We derive an analytic inversion approach based on circular harmonics expansion, and show uniqueness of reconstruction for the attenuated V-line transform. We further develop a discrete image reconstruction algorithm based on our analytic studies, and present numerical results that demonstrate the effectiveness of our algorithm.

\end{abstract}

\section{Introduction}
\label{sec:intro}

Single photon emission computed tomography (SPECT)
is a major medical diagnosis tool for functional imaging. Current SPECT
systems are based on the Anger camera for gamma ray detection \cite{anger1958scintillation}, which  uses collimators and typically records only one out of $10\,000$ actually emitted photons.  This yields to a large noise level  despite long recording times, and  consequently results in poor spatial and temporal resolution.  In order to increase the number of recorded photons, the concept of Compton cameras has been developed in~\cite{EveFleTidNig77,Sin83,TodNigEve74}.

The data in SPECT with Compton cameras consist of averages of the marker distribution over conical surfaces  or V-shaped lines.
Several  authors studied analytical approaches for image reconstruction  from Compton camera data~\cite{Allmaras13,basko1997fully,BasZenGul98,CreBon94,GouAmb13,Hal14a,hawkins1988circular,HirTom03,JunMoo15,Moon16a,moon2016analytic,Maxim2009,MorEtAl10,schiefeneder2016radon,Smi05,smith2011line,Terzioglu15}. However, in all these works the effect of the attenuation of photon has been neglected. As  in the case of standard SPECT, this can result in significant degradation of image quality.  In this work we establish a analytic reconstruction approach
for Compton camera imaging accounting for attenuation, which, to the best of our knowledge,  is the first in such a direction.

\begin{psfrags}
\psfrag{w}{\scriptsize $\psi$}
\psfrag{B}{\scriptsize $\beta$}
\psfrag{1}{(a)}
\psfrag{2}{(b)}
\psfrag{b}{\scriptsize $b$}
\psfrag{a}{\scriptsize $a$}
\psfrag{S}{\scriptsize $\sph_R$}
\psfrag{s}{\scriptsize $A$}
\psfrag{d}{\scriptsize $B$}

\begin{figure}[tb!]
\centering \includegraphics[width =\textwidth]{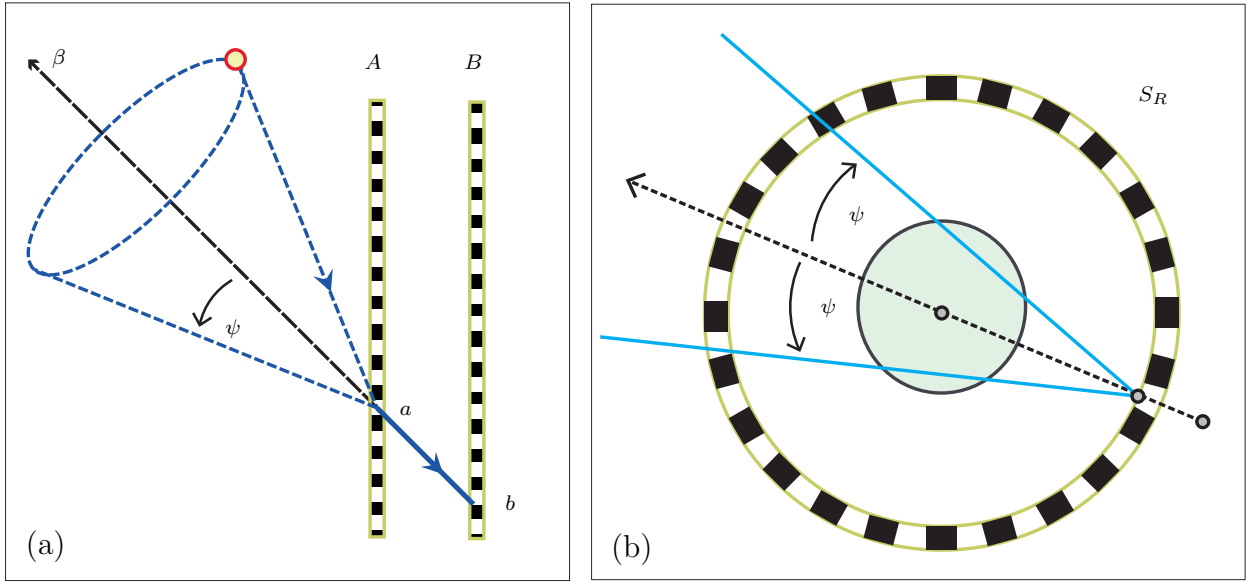}
\caption{(a) Compton camera consists of  two detector arrays  and any observed photon can be traced back to the surface of a cone.
(b) One-dimensional circular Compton camera. The conical integrals reduce to  integrals over V-shaped, with the vertex on $\sph_R$ and the symmetry axis pointing  to the origin.\label{fig:compton}}
\end{figure}
\end{psfrags}

\subsection{SPECT with Compton cameras}

In SPECT,  weakly radioactive tracers are given to a patient and are  detected through the  emission of gamma ray photons. Standard devices for photon detection in SPECT use Anger cameras based on collimation. These  kind of detectors only record photons that enter the detector vertically and therefore remove most photons. Compton cameras do not require a collimator  and in principle are capable of recording   all photons  that are emitted in the direction of the detector
array~\cite{EveFleTidNig77,Sin83,TodNigEve74}. Such devices consist of a scatter detector  array $A$ and an absorption detector array $B$.
As shown in Fig.~\ref{fig:compton}(a), a gamma ray photon arriving at the Compton camera undergoes Compton scattering
in $A$ and is absorbed in $B$.
Both detector arrays are position and energy sensitive,  and the measured energies can be used to determine the scattering angle via the Compton scattering formula.
One concludes  that the detected photon must have been emitted on the surface of a circular cone.
Consequently, for  a distribution of tracers, the Compton camera approximately provides
 integrals of  the marker
distribution over conical surfaces.

In this paper we consider two-dimensional Compton camera imaging,
where the conical surfaces reduce to V-lines with vertices on a circle; see Fig.~\ref{fig:compton}(b).
This two-dimensional version arises  for one-dimensional Compton camera proposed in \cite{BasZenGul97}  when the marker distribution is either supported in a plane or the detectors are collimated to this plane.

\subsection{The attenuated V-line transform}

We denote by $\D\coloneqq \sset{x\in \R^2 \mid \norm{x} < \RR }$  the unit disc in $\R^2$ and by $\sph_R$ the surrounding circle.
$C_c^\infty(\D)$ denotes the set  of all  smooth functions $\F \colon \R^2 \to \R$ with $\suppp (\F) \subseteq \D$.
For any angle $\ph \in \R$ we write $\spar(\ph)  \coloneqq (\cos\ph,\sin\ph)$ and $\spar(\ph)^\bot  \coloneqq (-\sin\ph, \cos\ph)$ such that
$(\spar(\ph) ,\spar(\ph)^\bot)$ forms a positive oriented orthonormal basis of $\R^2$. Further we denote by $\mu \in \R$ any attenuation value. In practice
we have $\mu\geq 0$. However since most of the following holds for general $\mu$,
we make restrictions on $\mu$ only when necessary.

\begin{definition}
The attenuated V-line transform $\Vo_\mu \F  \colon [0, 2\pi)  \times (0, \pi/2)   \to \R$ (with attenuation parameter $\mu$) of $\F \in C_c^\infty (\D)$ is defined by
\begin{equation}
  \label{eq:Vtrafo}
\Vo_\mu \F(\ph, \psi)  \coloneqq \sum_{\sigma = \pm 1}\int_0^\infty \F(\RR\spar(\ph)-r\spar(\ph - \sigma\psi)  )  \, e^{-\mu r} \, \dr \,.
\end{equation}
\end{definition}

The attenuated V-line transform  consists  of integrals of the emitter  distribution
over V-lines (the union of two half-lines) having the vertex $\RR \spar(\ph)$, symmetry axis  $\set{-r\spar(\ph) \mid r >0}$ and half opening angle $\psi$. The factor $e^{-\mu r}$ accounts for the attenuation of photons when propagating a distance $r$ in homogeneous media with attenuation value $\mu$.
In our numerical simulation studies we use
$\mu = \unit[0.15]{/cm}$, which is a realistic value
for soft tissue.

\subsection{Outline of main results}

In this paper we study the problem  of reconstructing the emitter distribution
$\F$ from  $\Vo_\mu \F$. Our main  contributions  can be summarized
as follows.

\begin{itemize}
\item[\ding{192}] {\scshape  Fourier series decomposition:}
Let $\F_n$ and $\G_n$ denote the Fourier  coefficients
with respect to the polar angle of $\F$ and  the vertex position of  $\Vo_\mu\F$, respectively.
In  Section~\ref{sec:fourier}, we derive an integral equation
for $\F_n$  in terms of    $\G_n$ with  an explicitly  given kernel of  the
generalized Abel type (see Eq.~\eqref{eq:V3}).

\item[\ding{193}] {\scshape  Uniqueness of reconstruction:}
In Section~\ref{sec:uni} we show that the generalized  Abel  equation \eqref{eq:V3}
has a unique solution and therefore $\F_n$ can be uniquely  recovered from $\G_n$ by solving~\eqref{eq:V3}; see Thm.~\ref{thm:uni}.
In particular, this implies  uniqueness of reconstruction of the attenuated V-line transform in the sense that  any data $\Vo_\mu \F$  corresponds  to exactly one
emitter distribution $\F$.

\item[\ding{194}] {\scshape  Numerical algorithm:} In
Section~\ref{sec:num} we derive a numerical algorithm for solving~\eqref{eq:V3}. Together with the FFT algorithm, this gives an efficient   Fourier discrete reconstruction algorithm for reconstructing $\F$  from $\Vo_\mu\F$. The proposed algorithm requires
only  $\mathcal{O} (\NN^2)$ floating point operations for reconstructing $\F$ at
 $\NN$ discretization points.

\item[\ding{195}] {\scshape  Numerical studies:}
In Section~\ref{sec:results} we present detailed numerical simulation studies
and demonstrate that our algorithm yields  accurate and fast reconstructions for
data with and without Poisson noise.
 For reconstructing  $\F$  at $40\, 000$ discretization points, our algorithm
 only requires about $1/40$ seconds on a standard PC.
We further demonstrate that our
 method is stable  with respect  to the selection of the   parameters involved.
 \end{itemize}

To the  best of our knowledge, these are the first  analytic results for Compton camera image reconstruction  accounting for  non-vanishing attenuation.

For other  Radon transforms, similar  approaches  have been previously and successfully applied  in \cite{Cor63,Ambarsoumian10,AmbMoo13,AmbRoy16,Hansen81circular,Per75,puro2001cormack,you1999unified}. However, in all  these works the arising integral equations satisfy standard conditions needed in order to apply standard well-posedness results for generalized Abel equations.
For \eqref{eq:V3}, one main assumption required for such results is violated, since the kernels turn out to have zeros on the diagonal. Nevertheless, by using a recent  result of  \cite{schiefeneder2016radon} we are able to establish solution uniqueness of the attenuated V-line transform.

\section{Fourier series decomposition}
\label{sec:fourier}

Throughout the following, $\mu \in \R$ denotes a fixed attenuation value.
Our inversion approach uses the Fourier
series with respect to the angular variables,
\begin{align} \label{eq:fn}
\F\kl{r\spar(\ph)}
&= \sum_{n \in \Z} \F_n(r)\,e^{\imi n\ph} \,,
\\  \label{eq:gn}
(\Vo_\mu\F)(\ph,\psi)    &=  \sum_{n \in \Z} \G_n(\psi)\,e^{\imi  n\ph}\,,
 \end{align}
where the Fourier coefficients  of the emitter  distribution  and the
corresponding attenuated V-line data  are defined by
\begin{align}\label{eq:fnn}
 \F_n(r)
 &\coloneqq  \frac{1}{2\pi}\int^{2\pi}_0 \F(r\spar(\ph))\,e^{-\imi n\ph}\rmd\ph \,,
 \\ \label{eq:gnn}
 \G_n(\psi)
 &\coloneqq \frac{1}{2\pi}\int^{2\pi}_0 (\Vo_\mu\F)(\ph,\psi)\,e^{-\imi  n\ph}\rmd{\ph} \,.
 \end{align}
Our strategy for inverting $\Vo_\mu$ is to recover each $\F_n$ from $\G_n$. For that purpose we derive a one-dimensional integral  equation for $\F_n$  in terms  of $\G_n$, which will
 subsequently be  solved theoretically and numerically.

We will  make use of the exponential Radon transform
$\To_{\mu} \F  \colon [0, 2\pi)  \times \R    \to \R$ defined by
\begin{equation}
(\To_{\mu} \F)(\al, s) \coloneqq \int_{\R} \F( s \spar(\al) + t \spar(\al)^\bot  )
\, e^{\mu  t} \, \dt \,.
\end{equation}
The exponential Radon transform integrates  the function $\F$ over the line
$\set{x \in \R^2 \mid  \inner{\spar(\al)}{x} = s}$  including the  weight  $e^{\mu t}$. It appears in image reconstruction  in SPECT with Anger cameras.
Its  inversion has been  addressed by many authors (see, for example, \cite{bellini79compensation,hawkins1988circular,inouye1989image,puro2001cormack,tretiak80exponential,yarman07new}).

\subsection{Auxiliary results}
\label{sec:aux}

We start our analysis by writing $\Vo_\mu\F$ as the sum of two exponential Radon transforms.

\begin{lemma}\label{lem:vr}
Suppose $\F \in C_c^\infty(\D)$. Then
\begin{equation}
	(\Vo_\mu\F)(\ph, \psi)
	= \label{eq:vr1}
	e^{-\RR \mu \cos(\psi)}
	\sum_{\sigma = \pm 1} (\To_{-\mu} \F) (\pi/2 + \ph-\sigma \psi, \sigma\RR \sin(\psi)) \,.
\end{equation}
\end{lemma}

\begin{proof}
With  $\al := \tfrac{\pi}{2} + \ph - \sigma\psi$  we have $\spar(\ph - \sigma\psi) = -\spar(\al)^\bot$.
A change of variables yields
\begin{align*}
&\int_0^\infty \F(\RR\spar(\ph)-r\spar(\ph - \sigma\psi)  )  \, e^{-\mu r} \, \dr
\\
&=
\int_{-\RR \cos(\psi)}^\infty \F(\sigma\RR \sin (\psi) \spar(\al)
+t \spar(\al)^\bot  )  \, e^{-\mu   (\RR \cos(\psi) +t )} \, \dt
\\
&=
e^{-\mu \RR \cos (\psi)} \int_{\R}  \F(\sigma\RR\sin (\psi) \spar(\al)
+t \spar(\al)^\bot  )  \, e^{- t \mu} \, \dt
\\&=
e^{-\mu \RR \cos (\psi)} (\To_{-\mu} \F) (\pi/2 + \ph - \sigma\psi,  \sigma\RR\sin(\psi)) \,.
\end{align*}
Together with the definition of $\Vo_\mu \F$, this gives~\eqref{eq:vr1}.
\end{proof}

\begin{lemma}\label{lem:Rn}
Let $\F \in C_c^\infty(\D)$ and denote by   $(\To_{-\mu}\F)_n (s) \coloneqq \frac{1}{2\pi}\int^{2\pi}_0 (\To_{-\mu}\F)(\al,s)\,e^{-in\al}\rmd{\al}$.
 Then,  for $(n, s) \in \Z \times \R$,
 \begin{equation} \label{eq:Rn1}
 (\To_{-\mu}\F)_n(s) = \sum_{\sigma = \pm 1} \int_{\sabs{s}}^R 
\F_n(r)\, e^{ \sigma \mu \sqrt{r^2-s^2}}
  e^{ - \imi n \sigma \arccos (s/r)}
\,  \frac{r}{\sqrt{r^2-s^2}} \, \dr \,.
\end{equation}
\end{lemma}

\begin{proof}
Using the definitions of the Fourier coefficients, the exponential Radon transform and the one-dimensional $\delta$-distribution, we obtain
\begin{align*}
&2\pi (\To_{-\mu} \F)_n(s)
\\
&=
\int_{0}^{2\pi}
\int_{\R} \F( s \spar(\al) + t \spar(\al)^\bot  )
\, e^{- \mu t} \, e^{- \imi n \al}\, \dt  \, \rmd \al
\\
&=
\int_{0}^{2\pi}
\int_{\R^2} \F(x) \delta(\inner{x}{\spar(\al)} -s)  \, e^{- \mu \inner{x}{\spar(\al)^\bot}} \, e^{- \imi n \al} \, \dx  \, \rmd \al
\\
&=
\int_{0}^{2\pi}
\int_{0}^{2\pi}
\int_{0}^{\infty}
\F(r\spar(\om)) \delta(r \cos (\om- \al)-s) \\
& \hspace{0.1\textwidth}  \times e^{- \mu r \sin (\om- \al)}  \, e^{- \imi n \al} \, r \, \dr  \, \dom \, \rmd \al
\\
&=
\int_{0}^{2\pi}
\int_{0}^{2\pi}
\int_{0}^{\infty}
\F(r\spar(\om)) \delta(r \cos (u)-s) \\
& \hspace{0.1\textwidth}  \times e^{ -\mu r\sin (u)}  \, e^{- \imi n \om}\, e^{ \imi n u} \, r \, \dr \, \dom \, \rmd u
\\
&=
2\pi \int_{0}^{2\pi}
\int_{0}^{\infty}
\F_n(r) \delta(r \cos (u)- s) \,  e^{-\mu r\sin (u)} \,
e^{ \imi n u} \, r \, \dr \, \rmd u
\,.
\end{align*}
Recall $\delta( g(u) )  = \sum_i \frac{\delta(u_i)}{\abs{g'(u_i)}}$, where the sum is taken over all simple zeros of $g$.   The zeros of $g(u) = r \cos (u)-s$ are given by $u_\pm = \pm  \arccos (s/r)$. They satisfy $\abs{g'(u_i)}
= r \abs{\sin (\arccos (s/r))} = \sqrt{r^2-s^2}$, which yields \eqref{eq:Rn1}.
\end{proof}

\subsection{Relation between the Fourier coefficients}
\label{sec:expansion}

In this section  we derive two different relations between $\F_n$ and $\G_n$.
The first one (Thm.~\ref{thm:Vn}) is   well suited
for the numerical  implementation, see Section~\ref{sec:num}.
The second one (Lem.~\ref{lem:glk3}) will be used  for
uniqueness of  reconstruction.

\begin{theorem}[Generalized Abel equation \label{thm:Vn} for  $\F_n$]
Suppose $\F \in C_c^\infty(\D)$, and let $\F_n$ and $\G_n$ for  $n\in \Z$  denote the Fourier coefficients of $\F$ and $\Vo_\mu \F$.
 Then,
 \begin{equation} \label{eq:Rn}
\forall \psi \in \bkl{0, \pi/2}\colon \quad \G_n(\psi)  = 2 e^{-\mu \RR \cos(\psi)}
 \int_{\RR\sin(\psi)}^R \F_n(r) \frac{ r \, \K_n(\RR\sin(\psi), r)  }{\sqrt{r^2-\RR^2\sin(\psi)^2}}  \, \rmd r \,,
\end{equation}
with the kernel function
\begin{equation} \label{eq:Rnk}
\K_n(s, r) :=
\sum_{\sigma =\pm1}  \sigma^n e^{\sigma  \mu \sqrt{r^2-s^2}}  \cos \kl{ n ( \arcsin (\tfrac{s}{r}) - \sigma \arcsin(\tfrac{s}{\RR}))}   \,.
\end{equation}
\end{theorem}

\begin{proof}
By Lem.~\ref{lem:vr} we have
\begin{align*}
  	& e^{\mu\RR\cos(\psi)} \G_n(\psi)
	\\
	&=
	 \sum_{\sigma = \pm 1}
	e^{\imi n (\pi/2 -\sigma \psi)} (\To_{-\mu} \F)_n (\sigma \RR\sin(\psi))
	\\
	&=
	 \imi^n \sum_{\sigma = \pm 1}
	e^{- \imi n \sigma \psi} (\To_{-\mu} \F)_n (\sigma \RR\sin(\psi))
	\,.
\end{align*}
Setting $s \coloneqq \RR\sin(\psi)$ and using Lemma \ref{lem:Rn},
we obtain
  \begin{align*}
  	 (- \imi )^n &e^{\mu \RR \cos(\psi)} \G_n(\psi)
	\\
	&=
	\sum_{\sigma_1, \sigma_2 = \pm 1}
	\int_{s}^R \Bigl[ \; \frac{ r \,\F_n(r)  }{\sqrt{r^2-s^2}}
	e^{-\imi n  \sigma_1 \psi}
	\\
	& \hspace{0.1\textwidth}
	\times  e^{ \sigma_2 \mu \sqrt{r^2-s^2}} e^{-\imi n \sigma_2  \arccos(\sigma_1 s/r) } \Bigr] \, \dr
	\\
	&=
	\sum_{\sigma_1, \sigma_2 = \pm 1}
	\int_{s}^R  \Bigl[ \; \frac{ r \,\F_n(r)  }{\sqrt{r^2-s^2}}
	e^{-\imi n  \sigma_1 \psi} \sigma_1^n
	\\
	& \hspace{0.1\textwidth}
	\times   e^{\sigma_2 \mu \sqrt{r^2-s^2}} e^{-\sigma_1 \sigma_2 \imi n  \arccos( s/r) } \Bigr] \dr
		\\
	&=
	(-\imi)^n
	\sum_{\sigma_1, \sigma_2 = \pm 1}
	\int_{s}^R  \Bigl[ \; \frac{ r \,\F_n(r)  }{\sqrt{r^2-s^2}}
	e^{-\imi n  \sigma_1 \psi}
	\\
	& \hspace{0.1\textwidth}
	\times  \sigma_2^n \;  e^{\sigma_2 \mu \sqrt{r^2-s^2}} e^{\sigma_1\sigma_2 \imi n  \arcsin( s/r) } \Bigr]\,  \dr
	\\
 &=
	2 \, (-\imi)^n
	\sum_{\sigma_2 = \pm 1}
	\int_{s}^R \Bigl[ \; \frac{ r \,\F_n(r)  }{\sqrt{r^2-s^2}}
	\sigma_2^n  e^{\mu \sigma_2 \sqrt{r^2-s^2}}
 	\\
	& \hspace{0.1\textwidth}
	\times
	 \cos ( n   \psi -\sigma_2  n  \arcsin( \tfrac{s}{r}) )
	 \Bigr] \,  \dr
 \,.
\end{align*}
Here the second and third equalities  follow from the identities
$\arccos(-x) = \pi - \arccos(x)$,
$\arccos(x) = \pi/2 - \arcsin(x)$.
The last equality shows \eqref{eq:Rn}, \eqref{eq:Rnk}.
\end{proof}

For the following alternative relation between  $\F_n$ and $\G_n$
we make use of the Chebyshev polynomials of the first  kind,
\begin{equation*}
	T_k(z) \coloneqq
	\cos \kl{k \arccos(z)}   \quad \text{ for } \abs{z} \leq 1 \,.
\end{equation*}
We then have the following  result.

\begin{lemma}\label{lem:glk3}
Let  $\F \in C_0^\infty (\D)$.
For $n \in \Z$,  let $\F_{n}$, $\G_n$ denote the Fourier  coefficients of $\F$, $\Vo_\mu\F$,
and write
\begin{enumerate}[label=(\alph*)]
\item\label{it:T1}
 $\hat{\G}_{n}(t)\coloneqq   \tfrac{1}{2} e^{\mu \RR \sqrt{t} } \G_n( \arccos (\sqrt{t}))$;
\item \label{it:T2} $\hat{\F}_{n}(\rho)\coloneqq   \RR  \, \F_n  ( \RR \sqrt{1-\rho} \, )$;
\item \label{it:T3} $\hatK_n(t,\rho) \coloneqq  \tfrac{1}{2} \sum_{\sigma = \pm 1} \sigma^n
 e^{\sigma \mu\RR \sqrt{t-\rho}} \, T_n \bigl(\frac{\sqrt{t}\sqrt{t-\rho} + \sigma (1-t)}{\sqrt{1-\rho}}\bigr) $.
\end{enumerate}
Then $\hat{\F}_{n}$ and $\hat{\G}_{n}$ are related via:
\begin{align}\label{eq:V3}
	\forall t \in [0,1] \colon \quad
	\hat{\G}_{n}(t)=\int_{0}^{t} \hat{\F}_{n}(\rho)
	\frac{\hatK_n (t,\rho)}{\sqrt{t-\rho}}\dd \rho \,.
\end{align}
\end{lemma}

\begin{proof}
By using the Chebyshev polynomials and using the trigonometric  sum
and difference identities we obtain
\begin{align*}
&\K_n(s, r) \\
& = \sum_{\sigma =\pm1}  \sigma^n e^{\sigma\mu \sqrt{r^2-s^2}} T_n \kl{  \cos( \arcsin (\tfrac{s}{r}) - \sigma \arcsin(\tfrac{s}{\RR}))} \\
& = \sum_{\sigma =\pm1}  \sigma^n e^{\sigma\mu \sqrt{r^2-s^2}} \,
T_n \bigl(  \cos( \arcsin (\tfrac{s}{r}))\cos( \arcsin (\tfrac{s}{\RR}))  \\
& \hspace{0.16\textwidth}     + \sigma \sin( \arcsin (\tfrac{s}{r})) \sin( \arcsin (\tfrac{s}{\RR})) \bigr)
\\
& = \sum_{\sigma =\pm1}  \sigma^n e^{\sigma\mu \sqrt{r^2-s^2}} \, T_n \Bigl(  \sqrt{1-\tfrac{s^2}{r^2} }    \sqrt{1-\tfrac{s^2}{\RR^2}}  + \sigma \tfrac{s}{r} \tfrac{s}{\RR} \Bigr)
\\
& = \sum_{\sigma =\pm1}  \sigma^n e^{\sigma\mu \sqrt{r^2-s^2}} \, T_n \Bigl(  \frac{\sqrt{r^2-s^2 }   \sqrt{\RR^2-s^2}  + \sigma s^2}{r \RR} \Bigr)  \,.
\end{align*}
Inserting the latter expression in \eqref{eq:Rn},
making the substitution $r^2 \gets R^2 - R^2 \rho$  and using~\ref{it:T1}
yields
 \begin{align*}
  &\hat{\G}_{n}(t)
  \\
  &= \tfrac{1}{2} \,  e^{\mu\RR \sqrt{t} } \G_n( \arccos (\sqrt{t}))
    \\
    &=
      \int_{\RR\sqrt{1-t} }^\RR
     \sum_{\sigma =\pm1}  \sigma^n
      \frac{ e^{\sigma\mu \sqrt{r^2-\RR^2+\RR^2 t }} }{\sqrt{r^2-\RR^2+\RR^2t}}
     \\
     & \hspace{0.01\textwidth} \times  T_n \Bigl(  \frac{\sqrt{r^2-\RR^2+\RR^2t} \,   \sqrt{\RR^2 t}  + \sigma \RR^2 (1-t)}{r \RR} \Bigr)
   \F_n(r)  r  \dr
    \\
    &=
      \tfrac{1}{2} \RR \int_{0}^t  \F_n(\RR \sqrt{1- \rho})
     \sum_{\sigma =\pm1}  \sigma^n e^{\sigma\RR \mu \sqrt{t-\rho}}
     \\
     & \qquad \times  T_n \Bigl(  \frac{ \sqrt{t-\rho }   \sqrt{t}  + \sigma (1-t)}{\sqrt{1-\rho}} \Bigr)
 \frac{  \rmd \rho }{\sqrt{t-\rho}}   \,.
 \end{align*}
With the definitions \ref{it:T2}, \ref{it:T3}
this gives~\eqref{eq:V3}.
\end{proof}

\section{Uniqueness of reconstruction}
\label{sec:uni}

The integral equation~\eqref{eq:V3} is of generalized  Abel type.
On the diagonal, the kernel $\hatK_n$ takes the form
\begin{equation}\label{eq:kn}
k_n(t) \coloneqq  \hatK_n(t,t) =  T_n (\sqrt{1-t})\,.
\end{equation}
Since the Chebyshev polynomials have zeros  in  $[0,1]$, the same holds for the function $t \mapsto  k_n(t)$.
Consequently, standard theorems on well-posedness  do not apply to~\eqref{eq:V3},
because such results require a non-vanishing diagonal.

\subsection{General uniqueness  result}

In order to show solution  uniqueness of~\eqref{eq:V3},   we use the following result that has recently been obtained in~\cite{schiefeneder2016radon}.

\begin{lemma}[Uniqueness\label{lem:abel} of generalized Abel equations with zeros on the diagonal,~\cite{schiefeneder2016radon}]
Let the kernel $\hatK \colon \domm \to \R$, with $\domm \coloneqq \set{(t,\rho) \in [0,1]^2 \mid 0 \leq \rho \leq t\leq 1}$,
satisfy the following:
\begin{enumerate}[label=(K\arabic*),leftmargin=3em]
	\item\label{lem:abel-1} $\hatK\in C^3(\domm)$.
	\item\label{lem:abel-2}  The set of zeros $N (\hatK) \coloneqq  \sset{ t \in [0,1) \mid \hatK(t,t)=0}$
	is finite and consists of simple zeros of $\hatK(t,t)$.
	\item\label{lem:abel-3} For every $t \in N (\hatK)$, $(\beta_1,\beta_2)
	\coloneqq \nabla \hatK (t,t)$  satisfies
	\begin{equation}\label{eq:ineq:Abel}
		1+\frac{1}{2}\,\frac{\beta_1}{\beta_1+\beta_2}>0 \,.
	\end{equation}
\end{enumerate}
Then, for any $\hat{\G} \in C([0,1])$, the generalized Abel  equation
\begin{equation}\label{eq:Abel}
	\forall t \in [0,1] \colon \quad
	\hat{\G}(t) = \int_0^t \frac{\hatK(t,\rho)}{\sqrt{t-\rho}} \, \hat{\F}(\rho)\dd \rho
\end{equation}
has at most one solution $\F \in C([0,1])$.
\end{lemma}

\begin{proof}
See \cite[Thm.~3.4]{schiefeneder2016radon}.
\end{proof}

\subsection{Uniqueness of the attenuated V-line transform}

We now apply Theorem \ref{lem:abel} to show uniqueness of a solution of
equation~\eqref{eq:V3}.

\begin{theorem}[Uniqueness \label{thm:uni}  of recovering $\F_n$]
Suppose $\mu \RR \leq 3/2$.
For any $\F \in C_0^\infty (\D)$ and $n\in \Z$, the Fourier coefficient $\F_n$  can be recovered as the unique solution of~\eqref{eq:Rn}.
\end{theorem}

\begin{proof}
Let $\F \in C_0^\infty (\D)$ vanish outside a ball of Radius $\RR^2 - \RR^2 a^2$ with $a < 1$.
According to Lem.~\ref{lem:glk3}, it is sufficient to show that   \eqref{eq:V3} has a unique
solution. To show that this is the case, we apply Lem.~\ref{lem:abel} by verifying  that
$\hatK_n $  satisfies \ref{lem:abel-1}-\ref{lem:abel-3}.
Clearly, $\hatK_n$ is smooth for $ t \neq \rho$. Further,  $\hatK_n$ can be written as  power series   only containing even powers of $\sqrt{st-\rho}$. Consequently,  $\hatK_n$ is also smooth on $\set{(t,\rho) \in \domm \mid t = \rho}$, which shows \ref{lem:abel-1}.  Next,  recall $ k_n(t)   \coloneqq \hatK_n(t,t)
= T_n(\sqrt{1-t})$.
As  $T_n$  has a finite number of isolated and simple roots this implies \ref{lem:abel-2}.

It remains to verify~\ref{lem:abel-3}. For that purpose, let
$t_0 \in [a, 1)$ be a zero of $k_n$ and set  $(\beta_1,\beta_2) \coloneqq \nabla K_n(t_0,t_0)$.
Then \begin{equation}\label{eq:beta12}
	\beta_1+\beta_2
	=
	k_n'(t_0)=
	- \frac{1}{2 \sqrt{1-t_0}}
	T_n'\left( \sqrt{1-t_0}\, \right) \,.
\end{equation}
Next we compute $\beta_1 = (\beta_1+\beta_2)- \beta_2$. For small  $\eps$,
\begin{align*}
&2 \hatK_n(t_0,t_0-\eps)
\\
&=  \sum_{\sigma = \pm 1} \sigma^n
e^{\sigma \mu\RR \sqrt{\eps}} \, T_n \left(\frac{\sqrt{t_0}\sqrt{\eps}+ \sigma (1-t_0)}{\sqrt{1-t_0+\eps}}\right)
\\&=
\sum_{\sigma = \pm 1} \sigma^n
\bkl{ 1 + \sigma \mu\RR  \sqrt{\eps}  }
\Bigl( \frac{\sqrt{t_0 }T_n'(\sigma \sqrt{1-t_0}) }{ \sqrt{1-t_0}}  \sqrt{\eps}
\\& \quad
+ \Big[ \frac{t_0 T_n''(\sigma \sqrt{1-t_0}) }{2(1-t_0)} -\frac{\sigma T_n'(\sigma \sqrt{1-t_0})}{2\sqrt{1-t_0}} \Big] \eps
\Bigr) + \mathcal O(\eps^2)
\\&=
 \frac{1}{ \sqrt{1-t_0}}  \Bigl(  (2 \mu\RR\sqrt{t_0} - 1)  T_n'(\sqrt{1-t_0})
 \\
 & \qquad  +  \frac{t_0 T_n''(\sqrt{1-t_0})}{ \sqrt{1-t_0}} \Bigr) \eps
+ \mathcal O(\eps^2)  \,.
\end{align*}
Here for the last equality we used $T_n'(-x)=(-1)^{n+1} T_n'(x)$ and $T_n''(-x)=(-1)^n T_n''(x)$.
Because $T_n$ is a solution of the differential equation
$(1-x^2)\,T_n''(x) -  \, x\, T_n'(x)+n^2\,T_n(x)=0$
and $t_0$ is a zero of $t \mapsto T_n(\sqrt{1-t})$,  it follows that
$    -\beta_2 =   \tfrac{ \mu\RR \sqrt{t_0}}{\sqrt{1-t_0}}\,
    T_n'(\sqrt{1-t_0})$.
Together with \eqref{eq:beta12} we obtain
\begin{equation}\label{eq:beta1}
\beta_1
=\frac{2 \mu\RR \sqrt{t_0} -1  }{\sqrt{1-t_0}\, }  \, T_n'(\sqrt{1-t_0}) \,.
\end{equation}
From \eqref{eq:beta12} and  \eqref{eq:beta1} we finally conclude
\begin{equation*}
1+ \frac{1}{2} \frac{\beta_1}{\beta_1 + \beta_2}
=
1- \frac{2 \mu\RR \sqrt{t_0} - 1}{2}
= \frac{3}{2} - \mu\RR \sqrt{t_0}  > 0 \,,
\end{equation*}
which is \ref{lem:abel-3}. Consequently,  Lem.~\ref{lem:abel} implies that
$\hat \F_n$ is the unique  solution of~\eqref{lem:glk3}.
\end{proof}

As a corollary of Thm.~\ref{thm:uni} we immediately obtain the
following uniqueness result for the attenuated V-line transform.

\begin{corollary}[Invertibility\label{cor:uni} of  $\Vo_\mu$]
Suppose $\mu\RR \leq 3/2$.
If $\F_1, \F_2   \in C_0^\infty (\D)$ satisfy
$\Vo_\mu \F_1 = \Vo_\mu \F_2$, then $\F_1 = \F_2$.
\end{corollary}

\begin{proof}
Let  $\F  \in C_0^\infty (\D)$ satisfy $(\Vo_\mu \F)_n = 0$ for all $n \in \Z$.
Thm.~\ref{thm:uni} shows that~\eqref{eq:Rn} has the
unique solution   $\F_n = 0$, which implies  $\F= 0$.
The linearity of $\Vo_\mu$ gives
the claim.
\end{proof}

For the case of vanishing attenuation, in  \cite{moon2016analytic} we derived an explicit  solution formula for~\eqref{eq:Rn}.
We have not been able to derive a similar  result $\mu \neq 0$; currently we don't know whether such a solution formula exists.
In the following section we show that~\eqref{eq:Rn} can be efficiently  solved numerically.

\section{Numerical reconstruction algorithm}
\label{sec:num}

In this section  we numerically implement the Fourier series  approach.
Suppose we have given discrete data
\begin{equation} \label{eq:data}
     \gn[\pp,\qq] \simeq  \Vo_\mu \F\kl{ \ph_\pp, \arcsin(s_\qq/\RR)}
      \quad
      \text{for }  (\pp,\qq) \in
      \set{0, \dots,  \PP-1} \times \set{0, \dots, \QQ} \,.
 \end{equation}
 Here $\ph_\pp \coloneqq   2\pi\pp/\PP$ and $s_\qq \coloneqq  \qq \RR / \QQ $  correspond to discrete vertex positions and  half opening angles.
 The goal is to estimate the values $\F(x_\ii)$ of the  emission distribution
 at grid points $x_\ii  =  (\ii_1, \ii_2)  \RR / \MM$
 for  $\ii=  (\ii_1 , \ii_2) \in \set{- \MM, \dots, \MM }^2$.

\subsection{Basic reconstruction strategy}

Thm.~\ref{thm:Vn} shows that $\F$ can be recovered   from $\Vo_\mu \F$
by implementing the following steps:
\begin{enumerate}[label=(S\arabic*),leftmargin=3em]
\item \label{alg:S1}
Evaluate $\G_n(\psi) \coloneqq \int_0^{2\pi} (\Vo_\mu\F)(\ph, \psi) e^{-\imi n \ph }\dd\ph$.
\item  \label{alg:S2}
Estimate  $\F_n$ by solving~\eqref{eq:Rn}.

\item \label{alg:S3}
Evaluate $\F( r \spar(\ph) ) = \sum_{n \in \Z}  \F_n(r) e^{\imi n \ph}$.

\item \label{alg:S4}
Resample $ \F$ to  Cartesian coordinates.
\end{enumerate}

As described in the following, in our numerical implementation we discretize any of these  steps.  For \ref{alg:S1} and~\ref{alg:S3}, we use the standard FFT algorithm.  For \ref{alg:S1}, the FFT algorithm outputs approximations to $\G_n$, which are used as inputs for \ref{alg:S2}. After implementing~\ref{alg:S3}, we have an approximation of $\F$ given on a polar grid.
For resampling  these values to a Cartesian grid, we use bilinear interpolation
in the polar coordinate space.

The main issue  in the reconstruction procedure consists in
solving the integral equation~\eqref{eq:Rn}. For that purpose we use the product integration
method with the mid-point rule~\cite{Linz,Plato12,Weiss71},
that is presented in the following subsection.

\subsection{The product integration  method for solving~\eqref{eq:Rn}}

Evaluating~\eqref{eq:Rn} at the discretization points $s_\qq$ yields
\begin{align*}
	\tilde \G_n(s_\qq)
& \coloneqq
\frac{e^{\mu \sqrt{\RR^2 - s_\qq^2}}}{2}
\G_n \bigl(\arcsin(\tfrac{s_\qq}{\RR}) \bigr)
	\\
&=
   \sum_{j=\qq}^{\QQ-1}
	\int_{s_j}^{s_{j+1}}
	\F_n(r)\,\frac{r  \K_n (s_\qq,r)}{\sqrt{\smash[b]{r^2-s_\qq^2}}} \, \dd r \,.
\end{align*}
 Approximating $\K_n(s_\qq,r) \simeq \K_n(s_\qq,r_j)$ on $r \in [s_j, s_{j+1}]$ where $r_j \coloneqq (j+1/2)\, \RR/\QQ$, we obtain
\begin{align*}
&\tilde \G_n(s_\qq)
	\simeq
	\sum_{j=\qq}^{\QQ-1} w_{\qq,j}
	\K_n \kl{s_\qq,r_j}  \F_n \kl{r_j }\,,
\\
&w_{\qq,j}  \coloneqq  \int_{s_j}^{s_{j+1}} \frac{r}{\sqrt{\smash[b]{r^2 - s_\qq^2}}} \dd r= \sqrt{\smash[b]{s_{j+1}^2  - s_\qq^2 }}-\sqrt{\smash[b]{s_j^2 - s_\qq^2}}\,.
\end{align*}
Set $w_{\qq,j} =0$ for $j \geq \qq$ and define
\begin{itemize}
\item discrete kernels $\Kn_n = (w_{\qq,j} \, \K_n(s_\qq,r_j))_{\qq,j =0,\dots,\QQ-1} $;
\item discrete data $\tilde \gn_n = (\tilde \G_n(s_0), \dots, \tilde \G_n(s_{\QQ-1}) )^\trans $;
\item discrete unknowns $\fn_n = (\fn_n[0], \dots, \fn_n[\QQ-1] )^\trans$.
\end{itemize}
The product integration method   then consists  in solving the following
 system of linear equations:
\begin{equation} \label{eq:Minv}
\text{Find} \quad \fn_n \in \R^\QQ
\quad \text{such that} \quad
\tilde \gn_n  = \Kn_n  \fn_n   \,.
\end{equation}
The matrix $\Kn_n$ is triangular.
If $\Kn_n$ is non-singular, then
\eqref{eq:Minv} can efficiently be solved by
forward substitution.

\subsection{Tikhonov regularization}

Because the  kernel function $\K_n$ has zeros in the diagonal,
the matrix $\Kn_n$  has diagonal entries being  close to zero. As a consequence, solving    \eqref{eq:Minv} is numerically unstable. In order to obtain stable solutions, regularization methods  have to be applied. We
apply  Tikhonov regularization \cite{EngHanNeu96,Gro84,Han98,scherzer2009variational,TikArs77} for that purpose, where regularized solutions $\fn_n^\la $ are  defined  as solutions of the regularized normal equation
\begin{equation} \label{eq:tik}
 \left(\Kn_n^\trans  \Kn_n  + \la_n \In_\QQ \right) \fn_n^\la
 =\Kn_n^\trans \tilde \gn_n \,.
\end{equation}
Here $\In_\QQ$ is the $\QQ \times \QQ$ identity matrix  and $\la = (\la_n)_n$
a vector of non-negative regularization parameters.

For a non-vanishing kernel diagonal, the  product integration  method \eqref{eq:Minv} is known to be convergent of order $3/2$; see \cite[Theorem 3.5]{Weiss71}.
Due to the zeros of the kernels, such results cannot be applied to the attenuated V-line  transform. We are not aware  of any results in that direction.  The numerical results indicate that
for suitable selection of the  regularization parameter, a convergence
analysis should be possible.

\subsection{Summary of reconstruction algorithm}

In summary, we  obtain the following
reconstruction algorithm for inverting the attenuated V-line transform.

\begin{framed}
\begin{alg}[Numerical inversion of the attenuated V-line transform]\mbox{}\\
\ul{Input}:\label{alg:Vnum} Data $\gn = (\gn[\pp,\qq])_{\pp,\qq} \in \R^{\PP \times (\QQ+1)}$; see \eqref{eq:data}.

\noindent \ul{Output}:  $\fn^\la \simeq (\F(x_\ii))_\ii \in \R^{(2\MM+1) \times (2\MM+1)}$.

\begin{enumerate}[label=(N\arabic*),leftmargin=2.5em]

\item \label{alg:N1}
Apply the FFT: $(\gn_n)_{n} := \operatorname{FFT} (\gn)$ .

\item  \label{alg:N2}
For any $n = -\PP/2, \dots , \PP/2-1$ do:
\begin{itemize}[leftmargin=0em]
\item Set $\tilde \G_n(s_\qq) \coloneqq
\tfrac{1}{2} \exp(\mu \sqrt{\smash[b]{\RR^2 - s_\qq^2}}\, )
\G_n(\arcsin(\tfrac{s_\qq}{\RR}))$;
\item Choose regularization parameters $\la_n > 0$;
\item Compute $\fn_n^\la$ by solving \eqref{eq:tik}.
\end{itemize}

\item \label{alg:N3}
Apply the inverse FFT: $\fn^\la_{\rm POL} := \operatorname{IFFT} ((\fn_n^\la)_n )$ .

\item \label{alg:N4}
Resample  $\fn^\la_{\rm POL}$ to a Cartesian grid.
\end{enumerate}
\end{alg}
\end{framed}

Steps \ref{alg:N1} and~\ref{alg:N3} in Alg.~\ref{alg:Vnum} consist of $\QQ$  one-dimensional FFTs and therefore require $\mathcal{O}(\QQ\PP\log \PP)$ floating point operations (FLOPS).   Using bilinear interpolation \ref{alg:N4} requires $\mathcal{O}(\MM^2)$ FLOPS. The most time consuming part is  \ref{alg:N2} which consists
of solving the $\PP$ linear equations \eqref{eq:tik},  each with $\QQ$ unknowns.
Using  the Cholesky decomposition these equations are solved  with approximately  $\QQ^3/6$ FLOPS. Supposing $\PP, \QQ , \MM = \mathcal O( \NN^{1/2})$, where $\NN$ is the total number of  unknowns,  the numerical effort of Alg.~\ref{alg:Vnum} therefore  is
$\mathcal O(\NN^2)$ with a small leading constant.
 On a standard PC, our algorithm requires about $1/40$ seconds for recovering about  $\NN
 = 40\, 000$  unknowns.

\begin{figure}[tb!]
\centering
\includegraphics[width =\textwidth]{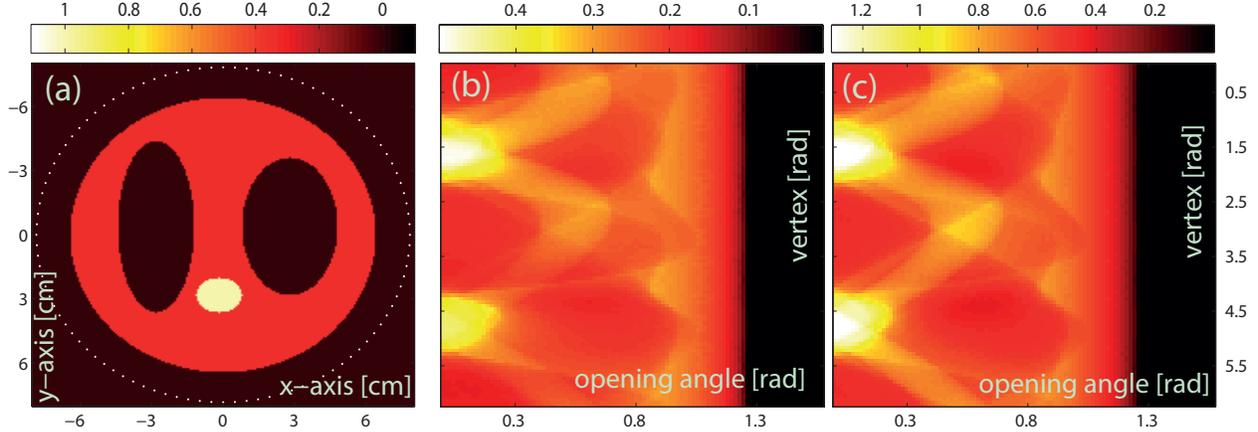}
\caption{{\scshape Phantom and data}. (a) Emission distribution (phantom) evaluated at a uniform $(2\MM+1) \times (2\MM+1)$ grid used for simulation studies. (b) Attenuated V-line transform with $\mu = \unit[0.15]{/cm}$ (c) Un-attenuated V-line transform.\label{fig:phant}}
\end{figure}

\section{Numerical results}
\label{sec:results}

For  the following numerical results we use the true emission distribution (phantom)
$\fn^\star$ shown in  Fig.~\ref{fig:phant}(a). It is contained in the disc of radius $\RR = \unit[8]{cm}$ and represented by discrete values on an $(2\MM+1) \times (2\MM+1)$ grid with $\MM \coloneqq 100$.
Attenuated V-line data are  simulated for $\PP = 100$ vertex positions  indicated by white  dots in Fig.~\ref{fig:phant}(a). At every vertex position we evaluate the attenuated V-line transform for $\QQ+1$ half opening angles with $\QQ \coloneqq 100$. The attenuation coefficient is taken as
$\mu =  \unit[0.15]{/cm}$.

\begin{figure}[tb!]
\centering
\includegraphics[width =0.9\textwidth]{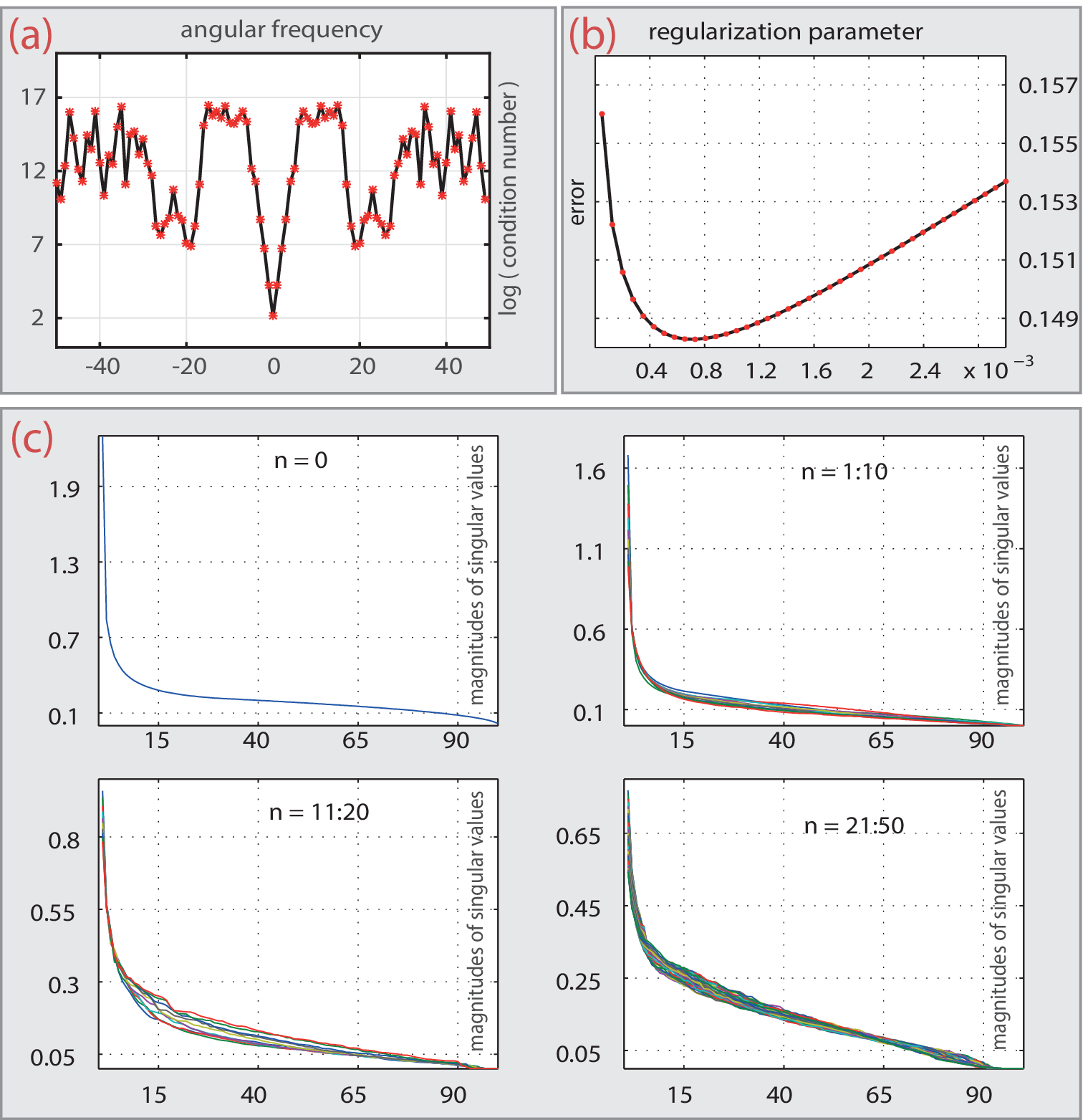}
\caption{{\scshape Investigation of stability}.  (a) Condition numbers of $\Kn_n$. (b) Relative $\ell^2$-reconstruction error as a function of the regularization parameter. (c) Singular values of $\Kn_n$ for $n=0$ (top left), $n =1, \dots, 10$ (top right), $n =11, \dots, 20$ (bottom left), and $n =21, \dots, 50$ (bottom right). \label{fig:matrix}}
\end{figure}

\subsection{Data computation}

For numerically computing the attenuated V-line transform, each of the
two branches of the V-line with vertex $\RR\spar(\ph_\pp)$  and half opening angle $\psi_\qq$ is sampled at  $2\MM+1$  equidistant discretization points in the interval $[0, 2\RR]$. The approximate function values $\F_{\rm BL}(\RR\spar(\ph_\pp)- j \tfrac{2\RR}{\MM}   \spar(\ph_\pp - \sigma\psi_\qq)  ) $ are computed by bilinear interpolation. Given these  approximate function values, we find
\begin{equation}\label{eq:Vf-num}
 g[\pp, \qq] \coloneqq \frac{2\RR}{\MM} \sum_{\sigma = \pm 1} \sum_{j=0}^{\MM}  \F_{\rm BL}(\RR\spar(\ph_\pp)
- \tfrac{2\RR}{\MM} j  \spar(\ph_\pp - \sigma\psi_\qq)  )  \, e^{- 2 \mu  j \RR / \MM  }
\end{equation}
as an approximation to $\Vo_\mu \F(\ph_\pp, \psi_\qq)$.

The numerically computed  attenuated V-line data corresponding to the phantom of  Fig.~\ref{fig:phant}(a) are shown in Fig.~\ref{fig:phant}(b). For comparison purpose Fig.~\ref{fig:phant}(c) shows the  V-line transform of the same phantom computed with attenuation value zero. One clearly notes  two effects of  attenuation: First, compared to unattenuated data, the overall intensity of the data is reduced. Second, and more importantly, the attenuation effects change  non-uniformly  over the data domain, which makes attenuation correction a non-trivial issue.

\begin{figure}[tb!]
\centering
\includegraphics[width =0.5\textwidth]{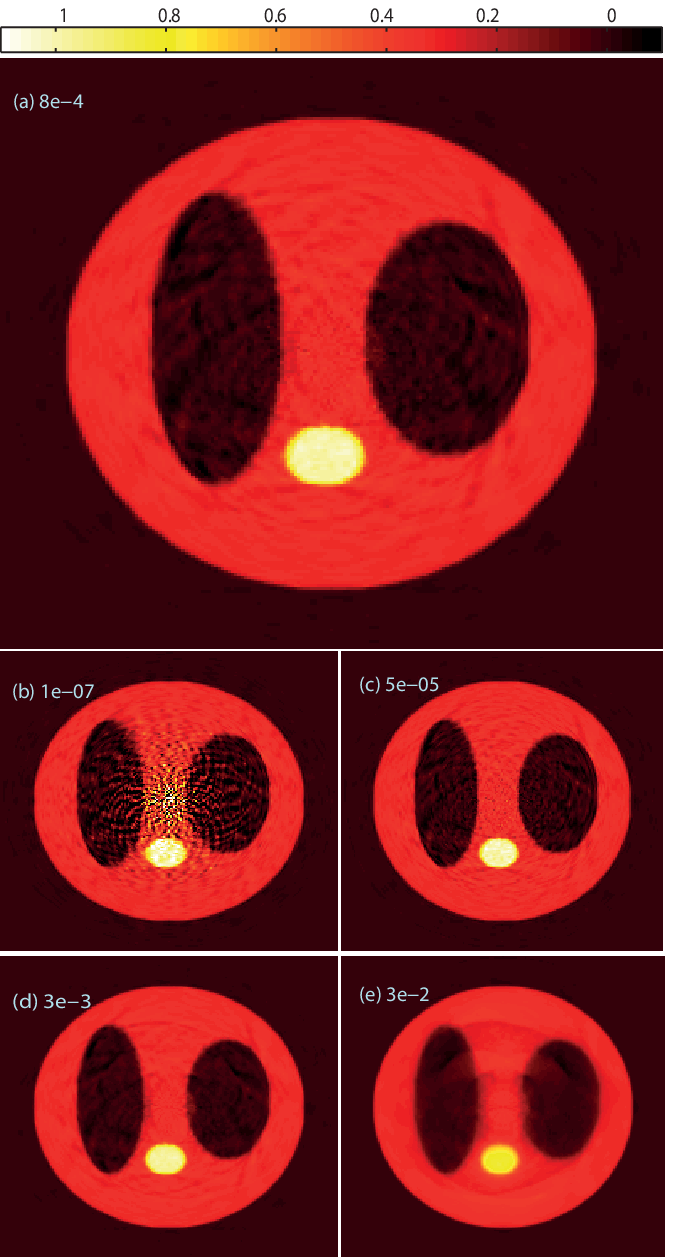}
\caption{{\scshape  Reconstructions from simulated data.} The regularization parameter has been chosen close to optimal value in (a), very small in (b), small in (c), large in (d) and very large in (e).\label{fig:recon}}
\end{figure}

\subsection{Reconstruction results for simulated data}

Next we present results of  Alg.~\ref{alg:Vnum} applied to  the
data shown in Fig.~\ref{fig:phant}.
The first issue  that has to be addressed is the selection of the regularization  parameter.
For that purpose, Fig.~\ref{fig:matrix}(a) displays
the condition numbers $\kappa(\Kn_n) \coloneqq  \snorm{\Kn_n}\, \snorm{\Kn_n^{-1}}$, which are a measure for the instability of solving \eqref{eq:Minv}. Except for $n=0$, the condition numbers   are large. Hence we
stabilize  any of the equations \eqref{eq:Minv} except the one for  $n = 0$.
To get more  inside in the instability of~\eqref{eq:Minv}, Fig.~\ref{fig:matrix}(c)  shows the singular values of $\Kn_n$.
For any of the matrices  with $n \neq 0$,
one observes a quite similar behavior. Therefore we use
a constant  positive regularization parameter  for $n\neq 0$.
Such a choice turns out to perform well in our  numerical studies.
Selecting  the  scalar regularization parameter still  is a non-trivial issue.
In the current studies we  have chosen it empirically by testing different values.
Data driven strategies such as  the discrepancy principle or the L-curve method \cite{EngHanNeu96,Han98}  will be investigated  in future studies.

For Fig.~\ref{fig:matrix}(b)  we compute the relative  $\ell^2$-reconstruction errors  $\snorm{\fn^\star- \fn^\la}_2/\snorm{\fn^\star}_2$
for different values of the regularization parameter $\la$.
One observes the typical semi-convergence behavior expected for ill-posed problems:  Starting with a large regularization parameter $\lambda$, the  error first decreases with decreasing $\la$ up to an optimal $\la^\star$. A further decrease of $\la$ increases the error due to  over-fitting of the data. In the  present case,
the  optimal regularization parameter turns out to be $\la^\star  \simeq 0.0008$.  The corresponding  reconstruction result is shown in Fig.~\ref{fig:recon}(a).

\begin{figure}[tb!]
\centering
\includegraphics[width =0.5\textwidth]{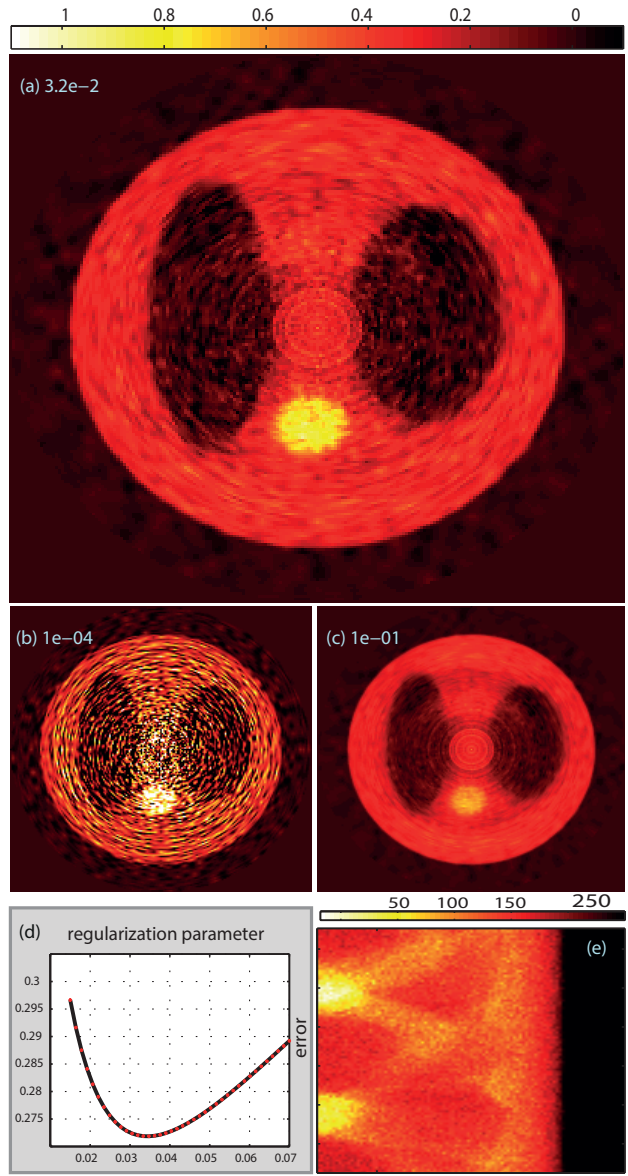}
\caption{{\scshape Noisy data simulations}.   (a) Reconstruction using optimal
$\la$.\label{fig:noisy}
(b) Reconstruction results using to small $\la$.
(c) Reconstruction results using to large  $\la$.
(d) Relative $\ell^2$ reconstruction error for  various $\la$
(e) Photon limited noisy data.}
\end{figure}

Fig.~\ref{fig:recon}(c) and Fig.~\ref{fig:recon}(d)  show reconstructions
for not optimally selected regularization parameters. The reconstruction results are still  good which demonstrates  the stability of  our algorithm  with respect to the choice of the regularization parameter.   In particular, a wide range  of parameters can be used to obtain accurate reconstructions. However,   $\la$ cannot be chosen  arbitrary far  away from the optimal value: Fig.~\ref{fig:recon}(b) corresponds to a small regularization parameter  far away from the optimal value.  High frequency error is evident.
Fig.~\ref{fig:recon}(e) shows results for a much too large  regularization parameter resulting  in a blurred reconstruction.

\subsection{Reconstruction results for photon limited data}

An extremely important  feature of any image reconstruction algorithm is its
ability  to deal with noisy data. In order to investigate this issue,  we performed simulations for limited number of photon counts. For the data shown in Fig.~\ref{fig:noisy}(e) we use a total number of $1\,894\,918$ photon counts and a maximal number of 573 photon counts on a single V-line.  The reconstruction results for different regularization  parameters   are shown in Figs.~\ref{fig:noisy}(a)-(c): In Fig.~\ref{fig:noisy}(a),
the regularization  parameter is chosen close to the optimal value, while in Fig.~\ref{fig:noisy}(b) it is chosen much too small, and in Fig.~\ref{fig:noisy}(c) it is chosen much too large. The same qualitative behavior as for simulated data can be observed.  However, as expected, the optimal regularization parameter is much larger and the reconstruction results worse than in the simulated data case.
The dependence on the regularization parameter has been investigated by computing the relative  $\ell^2$-error shown in Fig.~\ref{fig:noisy}(d).

\begin{figure}[tb!]
\centering
\includegraphics[width =\textwidth]{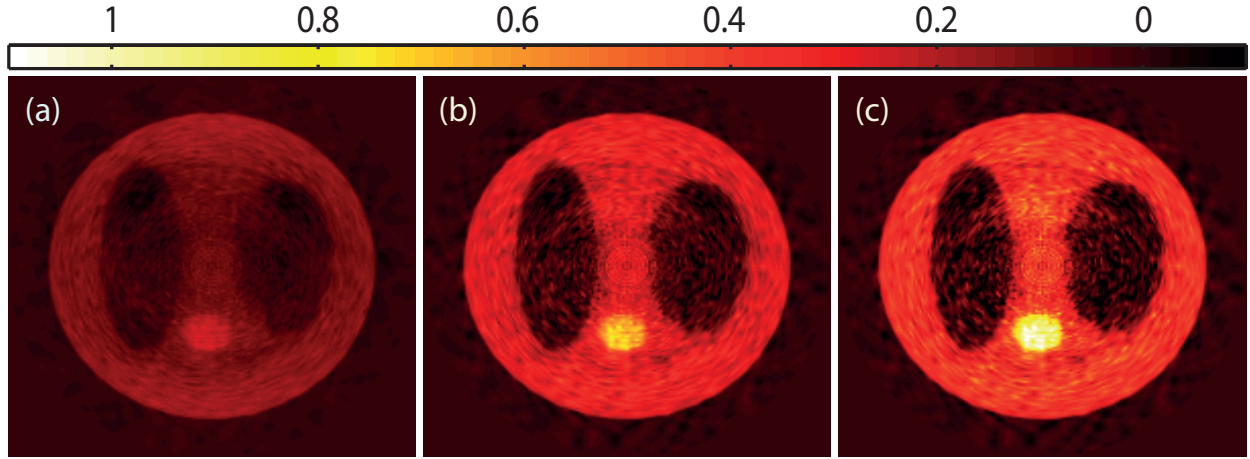}
\caption{{\scshape Stability with respect to correct attenuation value}.
(a) Reconstruction  from data shown in Fig.~\ref{fig:noisy} using Alg.~\ref{alg:Vnum} with $\la =0.03$ and assuming  vanishing   attenuation. (b) Same for under-estimated attenuation $\mu = \unit[0.125]{/cm}$.
(c) Same for  over-estimated  attenuation $\mu = \unit[0.175]{/cm}$. The  data is shown in Fig. \ref{fig:noisy} and has true attenuation value $\mu = \unit[0.15]{/cm}$.
\label{fig:stable}}
\end{figure}

In practical applications the attenuation value $\mu$ may not be known exactly.
We therefore apply Alg.~\ref{alg:Vnum} with  attenuation values that are different from the true attenuation. Results are shown in Fig.~\ref{fig:stable} where we used the same noisy data as above and a regularization parameter of $\la =0.03$.
One notices that the reconstruction results are quite stable with respect to the correct attenuation value. On the other hand, the poor reconstruction quality assuming vanishing attenuation shown in Fig.~\ref{fig:stable}(a) clearly demonstrates  that ignoring attenuation produces unacceptable results.

\section{Conclusion}
\label{sec:conclusion}

In this paper we  established a Fourier series approach for inverting the
attenuated V-line transform arising in SPECT with Compton  cameras.
We have been able to show invertibility of the attenuated V-line transform and to derive
an efficient reconstruction algorithm.

It is an interesting line of future work to generalize our results  in various directions.
For example, similar inversion approaches may be derived for various
geometries in two dimensions (attenuated V-line transforms) and three dimensions (attenuated conical Radon transform).
Further interesting extensions consider  the case of non-orthogonal
axis and non-constant attenuation. Finally, comparing our approach with iterative methods and testing on real data are important future aspects.

\section*{Acknowledgment}

The work of S. Moon has been supported by the National Research Foundation of Korea grant funded by the Korea government (MSIP) (2015R1C1A1A01051674) and the TJ Park
Science Fellowship of POSCO TJ Park Foundation.
S. Moon thanks the University of Innsbruck for  hospitality during his visit, when parts of this  work have been carried out.


\end{document}